\documentclass[11pt,a4paper]{article}
\usepackage{mathtools,amssymb,amsthm}
\usepackage{graphicx}
\usepackage[margin=2.5cm]{geometry}
\usepackage{enumitem}
\usepackage[section]{placeins}

\newtheorem{theorem}{Theorem}[section]
\newtheorem{proposition}[theorem]{Proposition}
\newtheorem{lemma}[theorem]{Lemma}
\newtheorem{corollary}[theorem]{Corollary}
\theoremstyle{definition}

\theoremstyle{remark}
\newtheorem{remark}[theorem]{Remark}

\newcommand{\R}{\mathbb{R}}
\newcommand{\C}{\mathbb{C}}
\newcommand{\Q}{\mathbb{Q}}
\newcommand{\Z}{\mathbb{Z}}
\newcommand{\dd}{\mathrm{d}}
\newcommand{\NF}{\mathrm{NF}}

\begin{document}

\title{On normal forms and first integrals of planar maps at elliptic fixed points}

\author{Stavros Anastassiou\\
Department of Mathematics\\
University of Western Macedonia\\
Kastoria, Greece\\
sanastassiou@gmail.com}

\date{}
\maketitle

\begin{abstract}
We study real-analytic planar maps, that admit a real-analytic first integral with a non-degenerate extremum, in a neighbourhood of an elliptic fixed point. We prove that, if the map preserves an analytic area form,  it is real-analytically and symplectically conjugate to a twist map, both in the resonant and in the non-resonant case. Using an example, we show that without the area-preservation assumption the conclusion fails. We then prove that a second, continuous, first integral is either a function of the first one or forces the map to be periodic, in which case the map is analytically (and, when area-preserving, symplectically) conjugate to its linear part. We illustrate our results, we study the McMillan map, and a planar map obtained from the Adler--Yamilov Yang--Baxter map.
\end{abstract}



\section{Introduction}

The qualitative analysis of real-analytic planar map germs near an elliptic fixed point has attracted a lot of attention. Since the pioneering work of Poincar\'e and Birkhoff, this analysis has relied heavily on formal normal forms to distinguish between regular, twist-like behaviour and chaotic dynamics.

This analysis is greatly simplified by the existence of first integrals, in which case the orbits of the map are contained in the level sets of the integral. Integrable maps of the plane are central objects in the theory of discrete integrable systems; the QRT family \cite{QRT1988,QRT1989,Duistermaat} is the prime example. For Hamiltonian vector fields, Zung \cite{Zung} proved that analytic integrability implies the convergence of the Birkhoff normalisation. For maps the situation is less understood: the available convergence results for integrable diffeomorphisms require some eigenvalue off the unit circle, or hyperbolicity-type conditions \cite{JiangStolovitch}, which exclude elliptic fixed points.

In this paper we consider a real-analytic planar map germ $f$ having an elliptic fixed point at the origin and admitting a real-analytic first integral $I_1$ with a non-degenerate extremum there. Our main results are the following.

\begin{enumerate}[label=(\alph*)]
\item If $f$ preserves an analytic area form, then $f$ is real-analytically and symplectically conjugate to a twist map $w\mapsto w\,e^{i\Omega(|w|^2)}$, for every rotation number, resonant or not (Theorem~\ref{thm:sympl}). The proof uses a real-analytic isochore Morse lemma (Lemma~\ref{lem:isochore}) and no convergence result for normal forms. In particular, the Birkhoff normalisation of such a map can be chosen convergent, and the resulting normal form contains no resonant terms.
\item Without area preservation the twist form cannot be achieved. To illustrate this, we give an explicit map whose restriction to the invariant curves of its analytic first integral is an Arnold circle map. We prove that this map is not topologically conjugate to any twist map, and its Poincar\'e normalisation diverges. What survives in general is a circle-preserving form (Proposition~\ref{prop:circle}).
\item A second continuous first integral $I_2$ is either a function of $I_1$ or forces $f$ to be periodic, $f^q=\mathrm{id}$; in the latter case $f$ is analytically conjugate to its linear part (Theorem~\ref{thm:main}). Consequently, for irrational rotation numbers no second independent first integral exists in any regularity class (continuous, smooth, real-meromorphic without pole at the origin, real-analytic), while for rational ones it exists if, and only if, $f$ is linearisable.
\item As an example, the McMillan map (a symmetric QRT map) illustrates both alternatives of (c) within a single one-parameter family (Section~\ref{sec:example}).
\item Yang--Baxter maps provide further examples. Starting from the Adler--Yamilov map, we perform two reductions on it to obtain a QRT-type planar map, whose elliptic fixed points are the plane waves of the reduced lattice equation $p_{n+1}-p_{n-1}=i\beta p_n/(1+|p_n|^2)$. All of them are elliptic and all rotation numbers occur. Near each of them, the map is conjugate to a twist map. There is no second first integral for irrational rotation numbers, nor at the resonances we check exactly (Section~\ref{sec:yb}).
\end{enumerate}

The paper is structured as follows. In Section~\ref{sec:NF} we recall the Poincar\'e--Birkhoff normal form theorem and present the normal forms in the resonant/non-resonant and general/symplectic cases, in a unified manner. In Section~\ref{sec:one} we study maps with one first integral, and in Section~\ref{sec:two} maps with two. Section~\ref{sec:example} is devoted to the McMillan map, Section~\ref{sec:yb} to the reduction of the Adler--Yamilov map, and the last section contains our conclusions.

Our approach to normal form theory uses the so-called classical style. We recommend \cite{Arnold,Kuznetsov} for a textbook exposition of this approach in the case of general dynamical systems, and \cite{Birkhoff,Moser,SiegelMoser} for the Hamiltonian case. In \cite{SandersVerhulstMurdock} the $\mathfrak{sl}_2$-style is beautifully presented, although for flows. We would also like to point out the articles \cite{BridgesCushman,Gelfreich}, in which the area-preserving case is studied, \cite{Simo} for the study of stability of elliptic points and \cite{JinZhang} for an application to billiards.

\section{Poincar\'{e}--Birkhoff normal forms}\label{sec:NF}

Let $f:(\R^2,0)\to(\R^2,0)$ be the germ at the origin of a real-analytic map. The origin is supposed to be a fixed point of $f$, while the eigenvalues of its linear part there, $L=\dd_0f$, lie on the complex unit circle and are not equal to $\pm1$. Thus they can be written in the form $\mu=e^{2\pi i\omega_0}$ and $\bar\mu$, where $\omega_0\in(0,\tfrac12)$. Throughout this work, we shall refer to such a germ as ``a real-analytic planar map germ having an elliptic fixed point at the origin with rotation number $\omega_0$''.

It is convenient to complexify the real plane ($\R^2\simeq\C$, $z=x+iy$) and consider the complexified map $f=f(z,\bar z)$ (by abuse of notation, we use the same symbol). We choose the complex coordinate so that $L(z)=\mu z$. If the real map $f$ preserves the standard symplectic form $\dd x\wedge\dd y$, the complexified map preserves $\tfrac i2\dd z\wedge\dd\bar z$.

The following theorem is of central importance, in normal form theory.

\begin{theorem}[Poincar\'e--Birkhoff normal form theorem]\label{thm:PB}
Let $f$ be the germ of a real-analytic map with an elliptic fixed point at the origin. There exists a formal change of coordinates $\phi$, tangent to the identity, such that $f_\NF:=\phi\circ f\circ\phi^{-1}$ commutes with the linear part of $f$, i.e.\ $f_\NF\circ L=L\circ f_\NF$. If $f$ preserves $\dd x\wedge\dd y$, then $\phi$ can be chosen symplectic, and $f_\NF$ is a formal symplectic map germ.
\end{theorem}

In the general setting $f_\NF$ is called the \emph{Poincar\'e normal form} of $f$, and in the symplectic case the \emph{Birkhoff normal form}. To set the notation and for the shake of completeness, we now give explicit expressions for these normal forms in four cases.

\subsection*{The non-resonant case, $\omega_0\notin\Q$}

Writing $w=\phi(z,\bar z)=z+O(|z|^2)$ and expanding $f_\NF(w,\bar w)=\sum_{m,n\ge0}a_{mn}w^m\bar w^n$, the relation $f_\NF(\mu w,\bar\mu\bar w)=\mu f_\NF(w,\bar w)$ yields:
\[
a_{mn}\big(\mu^{m-n-1}-1\big)=0,\qquad m,n\ge0 .
\]
Since $\omega_0\notin\Q$, $\mu^{m-n-1}=1$ if, and only if, $m=n+1$, so only the monomials $w^{n+1}\bar w^n=w|w|^{2n}$ survive and $f_\NF(w,\bar w)=\mu w\big(1+\sum_{k\ge1}c_k|w|^{2k}\big)$. For an appropriate choice of $\beta_k$, we can write the term in the parentheses in exponential form, $1+\sum c_k|w|^{2k}=\exp\big(\sum\beta_k|w|^{2k}\big)$, $\beta_k=\alpha_k+i\gamma_k$ and we thus arrive at the \textbf{non-resonant Poincar\'e normal form}:
\begin{equation}\label{eq:nonres-P}
f_\NF(w,\bar w)=w\exp\Big(2\pi i\omega_0+\sum_{k=1}^\infty(\alpha_k+i\gamma_k)|w|^{2k}\Big).
\end{equation}
The constants $\gamma_k$ and $\alpha_k$ are usually referred to as ``twist'' and ``radial growth'' coefficients, respectively. In polar coordinates, $w=re^{i\theta}$, it is:
\begin{equation}\label{eq:nonres-P-polar}
f_\NF(r,\theta)=\Big(r\exp\Big(\sum_{k=1}^\infty\alpha_kr^{2k}\Big),\ \theta+2\pi\omega_0+\sum_{k=1}^\infty\gamma_kr^{2k}\Big).
\end{equation}

\subsection*{The non-resonant symplectic case}

Assume moreover that $f$ preserves the symplectic form $\dd x\wedge\dd y$, and write \eqref{eq:nonres-P} as $f_\NF=\mu w\exp(Q(|w|^2))$ with $Q(s)=\sum_{k\ge1}(\alpha_k+i\gamma_k)s^k$. Using $\partial_wQ=\bar wQ'$ and $\partial_{\bar w}Q=wQ'$, the Jacobian determinant:
\[
\det Df_\NF=\frac{\partial f_\NF}{\partial w}\frac{\partial\bar f_\NF}{\partial\bar w}-\frac{\partial f_\NF}{\partial\bar w}\frac{\partial\bar f_\NF}{\partial w}
\]
is computed to be:
\[
\det Df_\NF(w,\bar w)=\exp\Big(2\sum_{k=1}^\infty\alpha_k|w|^{2k}\Big)\Big(1+2\sum_{k=1}^\infty k\alpha_k|w|^{2k}\Big).
\]
The condition $\det Df_\NF\equiv1$ forces $\alpha_k=0$ for all $k\ge1$, and we obtain the \textbf{non-resonant Birkhoff normal form}:
\begin{equation}\label{eq:nonres-B}
f_\NF(w,\bar w)=w\exp\Big(i\Big[2\pi\omega_0+\sum_{k=1}^\infty\gamma_k|w|^{2k}\Big]\Big),\qquad f_\NF(r,\theta)=\Big(r,\ \theta+2\pi\omega_0+\sum_{k=1}^\infty\gamma_kr^{2k}\Big).
\end{equation}
Observe that it is a twist map, which is formally integrable with first integral the function $|w|^2$.

\subsection*{The resonant case, $\omega_0=p/q$}

Let now $\omega_0=p/q\in\Q\cap(0,\tfrac12)$ with $\gcd(p,q)=1$ (so $q\ge3$). Then $\mu^q=1$ and $\mu^{m-n-1}=1$ if and only if $m-n-1=jq$ for some $j\in\Z$. The non-resonant monomials can be removed, and we are left with:
\begin{enumerate}[label=\alph*)]
\item for $j=0$: $w^{n+1}\bar w^n=w|w|^{2n}$;
\item for $j\ge1$: $w^{jq+1}|w|^{2n}$; for $j=-k\le-1$: $\bar w^{kq-1}|w|^{2m}$, $m\ge0$.
\end{enumerate}
We can factor out $w$ from all of these monomials except from the factor $\bar w^{kq-1}$ (the case $m=0$), and obtain the \textbf{resonant Poincar\'e normal form}:
\begin{equation}\label{eq:res-P}
f_\NF(w,\bar w)=\mu w\exp\Big(\sum_{k=1}^\infty(\alpha_k+i\gamma_k)|w|^{2k}+\sum_{j=1}^\infty\big(A_j(|w|^2)w^{jq}+B_j(|w|^2)\bar w^{jq}\big)\Big)+\sum_{j=1}^\infty C_j\,\bar w^{jq-1},
\end{equation}
where $\alpha_k,\gamma_k\in\R$, $C_j\in\C$ are constants and $A_j,B_j$ are formal power series. In polar coordinates:
\begin{equation}\label{eq:res-P-polar}
f_\NF(r,\theta)=\Big(r\exp\Big(\sum_{k=1}^\infty\alpha_kr^{2k}+R_q(r,\theta)\Big),\ \theta+\frac{2\pi p}{q}+\sum_{k=1}^\infty\gamma_kr^{2k}+\Theta_q(r,\theta)\Big),
\end{equation}
with $R_q,\Theta_q$ formal series that are $2\pi/q$-periodic in $\theta$ and vanish at $r=0$.

\subsection*{The resonant symplectic case}

If, in addition, $f$ preserves $\dd x\wedge\dd y$, the condition $\det Df_\NF\equiv1$ no longer forces the radial terms to vanish, because of the resonant monomials. A convenient description is the following \cite{Gelfreich}: the normalising transformation can be chosen symplectic, so that:
\begin{equation}\label{eq:res-B}
f_\NF=R_{2\pi p/q}\circ\Phi^1_H ,
\end{equation}
where $R_{2\pi p/q}(w)=\mu w$, $\Phi^1_H$ is the time-one map of the Hamiltonian flow of a formal Hamiltonian $H$, and $H\circ R_{2\pi p/q}=H$, i.e.\ $H$ is a formal series in $|w|^2$, $w^q$ and $\bar w^q$.

Since $H\circ R_{2\pi p/q}=H$ and $H\circ\Phi^1_H=H$, the formal series $H$ is a formal first integral of $f_\NF$. In contrast to the non-resonant case, however, $H$ is generically not a function of $|w|^2$ alone (its leading resonant part contains $\operatorname{Re}(w^q)$), so the resonant Birkhoff normal form is not, in general, a twist map; moreover, in general neither the normalising transformation nor $H$ converges.

In the next section we shall see the implications the existence of a first integral has in these normal forms.
\section{Maps with one first integral}\label{sec:one}

To begin our study of maps having a single first integral, let us first observe that the nature of the map's fixed point imposes restrictions on the possible first integrals.

\begin{lemma}\label{lem:hessian}
Let $I$ be a $C^2$ first integral of $f$ near $0\in \mathbb{R}^2$. Then $\dd_0I=0$, and the Hessian matrix of $I$ at $0$ is either zero or definite.
\end{lemma}

\begin{proof}
Differentiating the relation $I\circ f=I$ and evaluate at $0$ to get $\dd_0I\circ(L-\mathrm{id})=0$; since $1$ is not an eigenvalue of $L$, it is $\dd_0I=0$.

The Hessian $\mathcal{H}$ then satisfies $\mathcal{H}(Lv,Lv)=\mathcal{H}(v,v)$. In complex coordinates, write $L(z)=\mu z$ and $\mathcal{H}=c|z|^2+\operatorname{Re}(b z^2)$; the invariance relation then gives us $b(\mu^2-1)=0$. Since $\mu^2\neq1$, we must have $b=0$.
\end{proof}

In the generic case, the constant $c$ above will be non-zero. Replacing $I$ by $-I$ if necessary, we shall therefore work under the following standing assumption:

\medskip
\noindent\textbf{(H)} \emph{$f$ is a real-analytic planar map germ, having an elliptic fixed point at the origin with rotation number $\omega_0\in(0,\tfrac12)$, and $I_1$ is a real-analytic first integral of $f$, with $I_1(0)=0$ and a non-degenerate minimum at $0$.}
\medskip

In what follows, we use the notation: $w=x+iy=re^{i\theta}$, $C_r=\{|w|=r\}$, $D_\varepsilon=\{|w|<\varepsilon\}$ and $\omega=\dd x\wedge\dd y=r\,\dd r\wedge\dd\theta$.

\begin{proposition}\label{prop:circle}
Assume that \textup{(H)} holds. There exists a real-analytic diffeomorphism $G_0$, defined in a neighbourhood of $0$, with $I_1\circ G_0^{-1}(w)=|w|^2$ and such that the mapping $g=G_0\circ f\circ G_0^{-1}$, in polar coordinates, reads as:
\[
g(r,\theta)=\big(r,\ \theta+\Xi(r,\theta)\big),
\]
where, for every small $r>0$, the map $g_r:\theta\mapsto\theta+\Xi(r,\theta)$ is an, orientation-preserving, real-analytic diffeomorphism of the circle $C_r$.
\end{proposition}

\begin{proof}
$G_0$ is provided by the real-analytic Morse lemma (see, for example, \cite{AGV}). Since $g$ preserves $|w|^2$, it maps every circle $C_r$ onto itself. Its linear part $\dd_0g$ preserves the quadratic form $|w|^2$ and has as eigenvalues the (non-zero) $\mu,\bar\mu$. Thus, it is a rotation and it follows that $g$, and hence every $g_r$, preserves orientation.
\end{proof}

In general, $\Xi$ cannot be made independent of $\theta$ (see Proposition~\ref{prop:counter} below). However, in case $f$ preserves an area form it can. To prove it, we need the following lemma. The holomorphic version of this lemma is due to Vey \cite{Vey}, and the $C^\infty$ version to Colin de Verdi\`ere and Vey \cite{CdVV}. Here, we need the real-analytic planar case and we give, for the reader's convenience, a proof of it.

\begin{lemma}\label{lem:isochore}
Let $I$ be a real-analytic function defined in a neighbourhood of $0 \in \mathbb{R}^2$, satisfying $I(0)=0$ and having a non-degenerate minimum at the origin. Then there exists a real-analytic diffeomorphism $S$ near $0$, with $S(0)=0$ and $S^*\omega = \omega$ (where $\omega = dx \wedge dy$ is the standard symplectic form), and a real-analytic function $F$ with $F(0)=0$ and $F'(0) > 0$, such that:
\[
I \circ S^{-1}(x,y) = F(x^2 + y^2).
\]
\end{lemma}

\begin{proof}
First, by the classical Analytic Morse Lemma, since $0$ is a non-degenerate critical point with a positive-definite Hessian matrix, there exists a local real-analytic coordinate chart $\psi$ near $0$ such that in the new coordinates $(\xi, \eta)$, the function takes the standard normal form:
\[I \circ \psi^{-1}(\xi, \eta) = \xi^2 + \eta^2.\]

We define the area function with respect to the standard symplectic form as:
\[A(h) = \int_{\{I < h\}} \omega = \iint_{\xi^2 + \eta^2 < h} J(\xi, \eta) \, d\xi \, d\eta,\]
where $J(\xi, \eta) = \vert{}\det D\psi^{-1}(\xi, \eta)\vert{}$. Because $\psi$ is a real-analytic mapping, its Jacobian $J(\xi, \eta)$ is a real-analytic function in a neighborhood of the origin, and thus it can be expanded into a globally convergent power series in two variables:
\[J(\xi, \eta) = \sum_{a,b \ge 0} j_{ab} \, \xi^a \eta^b.\]
To evaluate the integral, we switch to polar coordinates $(\xi = r\cos\theta, \eta = r\sin\theta)$:
\[A(h) = \int_0^{2\pi} \int_0^{\sqrt{h}} \left( \sum_{a,b \ge 0} j_{ab} \, (r\cos\theta)^a (r\sin\theta)^b \right) r \, dr \, d\theta.\]
By absolute convergence of the power series within its radius of convergence, we can interchange the sum and the integrals:
\[A(h) = \sum_{a,b \ge 0} j_{ab} \left( \int_0^{2\pi} \cos^a\theta \sin^b\theta \, d\theta \right) \left( \int_0^{\sqrt{h}} r^{a+b+1} \, dr \right).\]
Let us examine the angular integral factor:
\[I_{a,b} = \int_0^{2\pi} \cos^a\theta \sin^b\theta \, d\theta.\]
If either $a$ or $b$ is an \emph{odd} integer, we have $I_{a,b} = 0$. So, let $a = 2p$ and $b = 2q$ for some non-negative integers $p, q \ge 0$. For these non-vanishing terms, we compute the radial integral:
\[\int_0^{\sqrt{h}} r^{a+b+1} \, dr = \frac{h^{\frac{a+b}{2} + 1}}{a+b+2}.\]
Setting $k = \frac{a+b}{2}$ (which is an integer since $a, b$ are even), each surviving term contributes a constant multiple of $h^{k+1}$. This proves that $A(h)$ is not just a smooth function, but a \emph{convergent power series in $h$} of the form:
\[A(h) = \sum_{k=0}^{\infty} c_k h^{k+1} = c_0 h + c_1 h^2 + c_2 h^3 + \dots\]
It follows that $A(0) = 0$ and:
\[A'(h) = c_0 + 2c_1 h + 3c_2 h^2 + \dots\]
The leading coefficient $c_0$ corresponds to $a = 0$ and $b = 0$ (so $k = 0$), where $j_{00} = J(0,0) = J(0)$, thus, the lowest-order term is:
\[j_{00} \cdot (2\pi) \cdot \left(\frac{h}{0+0+2}\right) = \pi j_{00} h = \pi J(0) h.\]
Hence, evaluating the derivative at $h = 0$ yields:
\[A'(0) = c_0 = \pi J(0) = \pi \left\vert{}\det D\psi^{-1}(0)\right\vert{} > 0,\]
since $\psi$ is a local diffeomorphism, ensuring its Jacobian determinant is non-zero everywhere, and hence, strictly positive.

To normalize the period of the Hamiltonian flow of $I$ on each level set $\{I = h\}$, we introduce a new Hamiltonian function $K$ defined as the composition:
\[
K = A \circ I.
\]
Let us examine the geometric relationship between the Hamiltonian vector fields of $I$ and $K$. It is:
\[
\iota_{X_K} \omega = -dK = -A'(I) \, dI = A'(I) \, (\iota_{X_I} \omega) = \iota_{A'(I) X_I} \omega.
\]
Since the symplectic form $\omega$ is non-degenerate, this implies the vector field relation:
\[
X_K = A'(I) \, X_I.
\]
For a fixed small $h > 0$, let $T(h)$ denote the period of the closed integral curve (orbit) of the vector field $X_I$ on the level set surface $\{I = h\}$. It is $T(h) = A'(h)$ (see \cite{Arnold3}, $\S50$).

Now consider the flow $\varphi^t_K$ of $X_K$. Along any level set $\{I = h\} = \{K = A(h)\}$, the time parameter required to traverse the full closed orbit $\{I = h\}$ under the flow of $X_K$ is given by:
\[
T_{\text{new}} = \frac{T(h)}{A'(h)} = \frac{A'(h)}{A'(h)} = 1.
\]
Thus, \emph{every} trajectory of the vector field $X_K$ in a sufficiently small neighborhood of the origin $0$ has a uniform, constant period equal to $1$. 

Since the flow $\varphi^t_K$ is $1$-periodic for all points near $0$, the map $(t, w) \mapsto \varphi^t_K(w)$ descends to a real-analytic action of the circle group $\mathbb{R}/\mathbb{Z}$ (isomorphic to $S^1$) to a neighborhood of $0 \in \mathbb{R}^2$. Furthermore, every map $\varphi^t_K$ is a local symplectomorphism fixing the origin $0$.

We linearize the action at $0$. Let:
\[
R_t = d_0 \varphi^t_K : T_0\mathbb{R}^2 \to T_0\mathbb{R}^2.
\]
This linear flow corresponds to the Hamiltonian flow of the quadratic part of $K$ at the origin. Since $I$ has a non-degenerate minimum at $0$, $K = A(I)$ also possesses a non-degenerate minimum at $0$, meaning its Taylor expansion begins with a positive-definite quadratic form. 

A positive-definite quadratic form in two variables always has a $1$-periodic Hamiltonian flow. By standard linear symplectic algebra, any such positive-definite quadratic Hamiltonian can be brought, via a linear symplectic coordinate change, to the canonical form:
\[
\pi(x^2 + y^2).
\]
Under this linear symplectic change of coordinates, the linearized flow $R_t$ is transformed into the standard, rigid rotation of the Euclidean plane by the angle $2\pi t$:
\[
R_t(x, y) = (x\cos(2\pi t) - y\sin(2\pi t), \, x\sin(2\pi t) + y\cos(2\pi t)).
\]

To align the non-linear circle action $\varphi^t_K$ with its linear model $R_t$, we wish to find a local real-analytic diffeomorphism $B$, near $0$, that conjugates the two actions, meaning $B \circ \varphi^t_K = R_t \circ B$. To achieve this, we employ the classical Bochner averaging operator (or Reynolds operator for compact group actions \cite{Bochner}). 

We define the mapping $B$ for any vector $w$ in a sufficiently small neighborhood of $0 \in \mathbb{R}^2$ as follows:
\[
B(w) = \int_0^1 R_{-t} \big( \varphi^t_K(w) \big) \, dt.
\]
\begin{enumerate}
    \item This operator is analytic and:
    \[
    d_0 B = \int_0^1 R_{-t} \circ d_0(\varphi^t_K) \, dt = \int_0^1 R_{-t} \circ R_t \, dt = \int_0^1 \mathrm{id} \, dt = \mathrm{id}.
    \]
    \item For any fixed time $\tau \in \mathbb{R}/\mathbb{Z}$, we evaluate that:
    \[
    B(\varphi^\tau_K(w)) = \int_0^1 R_{-t} \big( \varphi^t_K(\varphi^\tau_K(w)) \big) \, dt = \int_0^1 R_{-t} \big( \varphi^{t+\tau}_K(w) \big) \, dt.\ (\star)
    \]
    Making a change of integration variable by setting $u = t + \tau$, and using the $1$-periodicity of the circle action, we obtain:
\begin{align*}
\star &= \int_\tau^{1+\tau} R_{- (u - \tau)} \big( \varphi^u_K(w) \big) \, du = \int_0^1 R_{\tau} R_{-u} \big( \varphi^u_K(w) \big) \, du=\\
&=R_\tau \left( \int_0^1 R_{-u} \big( \varphi^u_K(w) \big) \, du \right) = R_\tau \big( B(w) \big).
\end{align*}
    Thus, $B \circ \varphi^\tau_K = R_\tau \circ B$
\end{enumerate}

Next, we consider the pull-back of the standard symplectic form $\omega$ under the inverse of $B$, denoted as:
\[
\omega' = (B^{-1})^* \omega.
\]
Because $B$ commutes with the linear rotations and each $\varphi^{-t}_K$ is a symplectomorphism preserving $\omega$, the new symplectic form $\omega'$ is invariant under all rotation (meaning $R_t^* \omega' = \omega'$) and, furthermore, since $d_0 B = \mathrm{id}$, $\omega'(0) = \omega(0)$.

Now, both $\omega$ and $\omega'$ are closed $2$-forms on a contractible neighborhood of $0$ and agree at the origin. Thus, the difference $\omega' - \omega$ is an exact $2$-form. By the Poincaré Lemma, we can write:
\[
\omega' - \omega = d\beta,
\]
where the $1$-form $\beta$ can be constructed explicitly via the standard radial homotopy formula. Since the linear rotations $R_t$ commute with radial scaling and preserve both $\omega$ and $\omega'$, the resulting primitive $1$-form $\beta$ is also real-analytic, vanishes quadratically at the origin  and is strictly invariant under all rotation operators $R_t$ ($R_t^* \beta = \beta$).

We now use Moser's path method to find a local diffeomorphism that trivializes the exact difference $d\beta$. Define:
\[
\omega_s = \omega + s(\omega' - \omega), \quad s \in [0, 1].
\]
We seek a time-dependent vector field $X_s$ (with flow $\psi_s$) such that $\frac{d}{ds}(\psi_s^* \omega_s) = 0$. By Cartan's magic formula, this condition reduces to finding $X_s$ such that:
\[
\iota_{X_s} \omega_s = -\beta.
\]
Due to the fact that both $\omega_s$ and $\beta$ are invariant under the linear rotations $R_t$, the resulting vector field $X_s$ is also $R_t$-invariant and vanishes at the origin. 

The time-one flow map $\psi_1$ commutes with all linear rotations, because $X_s$ is $R_t$-invariant, and $\psi_1^* \omega' = \omega$. 

Finally, we define the map $S = \psi_1^{-1} \circ B$.

Let us check that:
\begin{enumerate}
    \item $S^* \omega = (\psi_1^{-1} \circ B)^* \omega = B^* (\psi_1^{-1})^* \omega = B^* \omega' = B^* (B^{-1})^* \omega = \omega$.
    \item $    S \circ \varphi^t_K = (\psi_1^{-1} \circ B) \circ \varphi^t_K = \psi_1^{-1} \circ (B \circ \varphi^t_K) = \psi_1^{-1} \circ (R_t \circ B) = R_t \circ (\psi_1^{-1} \circ B) = R_t \circ S$.
\end{enumerate}

Since $S$ is a symplectic diffeomorphism that conjugates the flow of $K$ to the flow of $\pi(x^2 + y^2)$, for the corresponding vector fields we have:
\[
S_* X_K = X_{\pi(x^2 + y^2)},
\]
and  thus:
\[
K \circ S^{-1}(x,y) = \pi(x^2 + y^2) + c,
\]
for some constant $c$, which must be equal to zero, since both sides vanish there. The last equation now becomes:
\[
I \circ S^{-1}(x,y) = A^{-1}\big(\pi(x^2 + y^2)\big).
\]
We define a new real-analytic function of a single real variable by setting:
\[
F(u) = A^{-1}(\pi u),
\]
and the proof is complete.
\end{proof}

\begin{theorem}\label{thm:sympl}
Assume that \textup{(H)} holds and that $f$ preserves the form $\omega=\dd x\wedge\dd y$. Then, there exists a real-analytic diffeomorphism $S$ near $0$, with $S^*\omega=\omega$ and such that:
\begin{equation}\label{eq:twist}
S\circ f\circ S^{-1}(w)=w\,e^{i\Omega(|w|^2)},\qquad\text{i.e.}\qquad S\circ f\circ S^{-1}(r,\theta)=\big(r,\ \theta+\Omega(r^2)\big),
\end{equation}
where $\Omega$ is real-analytic and $e^{i\Omega(0)}\in\{\mu,\bar\mu\}$. This holds for both rational and irrational $\omega_0$.
\end{theorem}

\begin{proof}
Let $S$ be given by Lemma~\ref{lem:isochore} applied to $I_1$, and set $g=S\circ f\circ S^{-1}$. Then $g$ preserves $F(|w|^2)$ and, since $F$ is strictly increasing near $0$, also $|w|^2$. As in Proposition~\ref{prop:circle}, $g(r,\theta)=(r,\theta+\Xi(r,\theta))$, and:
\[
g^*(r\,\dd r\wedge\dd\theta)=r\big(1+\partial_\theta\Xi\big)\,\dd r\wedge\dd\theta .
\]
Since $g^*\omega=\omega$, we must have $\partial_\theta\Xi\equiv0$, i.e.\ $\Xi=\Xi(r)$.

It remains to show that $\Xi(r)=\Omega(r^2)$ with $\Omega$ real-analytic. The function $k(x)=g(x,0)/x$ (in complex notation) is real-analytic near $x=0$ and $|k|=1$. For $x>0$ we have $k(x)=e^{i\Xi(x)}$. For $x<0$ we have $g(x,0)=|x|e^{i(\pi+\Xi(|x|))}=x\,e^{i\Xi(|x|)}$, so $k(x)=e^{i\Xi(|x|)}$. Thus $k$ is even, $k(x)=\kappa(x^2)$ with $\kappa$ real-analytic, and $\Omega=-i\log\kappa$ (continuous branch) is real-analytic with $e^{i\Xi(r)}=e^{i\Omega(r^2)}$. Finally, $\dd_0g$ is conjugate to $L$, so $e^{i\Omega(0)}\in\{\mu,\bar\mu\}$.
\end{proof}

\begin{remark}\label{rem:area}
\begin{enumerate}[label=(\roman*)]
\item The twist map \eqref{eq:twist} commutes with all rotations, thus, after a, linear, symplectic change of coordinates, it is a Birkhoff normal form of $f$. Theorem~\ref{thm:sympl}, therefore, shows that, for an area-preserving map with an analytic first integral having a non-degenerate extremum, the Birkhoff normalisation can be chosen convergent. When $\omega_0=p/q$, all resonant terms in \eqref{eq:res-B} can then be removed. This is a planar, discrete analogue of Zung's theorem \cite{Zung}, obtained without normal form techniques.
\item The theorem applies verbatim to maps preserving an analytic area form $\rho\,\dd x\wedge\dd y$ with $\rho(0)\neq0$. Indeed, $\Psi(x,y)=\big(x,\int_0^y\rho(x,t)\,\dd t\big)$ is a local analytic diffeomorphism with $\Psi^*(\dd x\wedge\dd y)=\rho\,\dd x\wedge\dd y$, so $\Psi\circ f\circ\Psi^{-1}$ preserves $\dd x\wedge\dd y$. This covers, for instance, QRT maps \cite{QRT1989,Duistermaat} at elliptic fixed points where their invariant density does not vanish.
\end{enumerate}
\end{remark}

We would like to underline that Theorem~\ref{thm:sympl} does not extend to maps which do not preserve an area form. Proposition~\ref{prop:circle} is, therefore, the best one can expect, in the general case. 

\begin{proposition}\label{prop:counter}
Let $\omega_0\in(0,\tfrac12)\setminus\Q$, $\alpha=2\pi\omega_0$, and:
\[
g(w)=w\,\exp\big(i(\alpha+\operatorname{Re}w)\big).
\]
Then $g$ is real-analytic, $g(0)=0$, $\dd_0g$ is multiplication by $e^{i\alpha}$, and $I_1(w)=|w|^2$ is a real-analytic first integral with a non-degenerate minimum at $0$. Nevertheless, $g$ is not topologically conjugate, on any neighbourhood of $0$, to a map $(r,\theta)\mapsto(r,\theta+\Omega(r^2))$ with $\Omega$ continuous. In particular, $g$ admits no convergent normalising transformation to its Poincar\'e normal form.
\end{proposition}

\begin{proof}
The first assertions are immediate.

The angular dynamics on a circle $C_r$ of radius $r$ are governed by the Arnold circle map:
\[
g_r(\theta) = \theta + \alpha + r\cos\theta.
\]
Let $\rho(r)$ denote the rotation number of $g_r$. Since $\rho(r)$ depends continuously on the parameter $r$ and $\rho(0)=\omega_0$, we aim to compute its asymptotic expansion near $r=0$. 
We seek a near-identity smooth change of coordinates on the circle of the form:
\[
h_r(\theta) = \theta + r\,u(\theta) + r^2v(\theta) + O(r^3),
\]
where $u$ and $v$ are smooth, $2\pi$-periodic, functions chosen such that the conjugated map $h_r \circ g_r \circ h_r^{-1}$ assumes the normal form rotation up to order $r^2$:
\[
h_r(g_r(\theta)) = h_r(\theta) + \alpha + \beta_2 r^2 + O(r^3).
\]
Expanding the left-hand side, we get:
\begin{align*}
h_r(g_r(\theta)) &= (\theta + \alpha + r\cos\theta) + r\,u(\theta + \alpha + r\cos\theta) + r^2 v(\theta + \alpha) + O(r^3) \\
&= \theta + \alpha + r\cos\theta + r\left[u(\theta) + u'(\theta)r\cos\theta\right] + r^2 v(\theta + \alpha) + O(r^3).
\end{align*}
Equating it to $h_r(\theta) + \alpha + \beta_2 r^2 = \theta + \alpha + r u(\theta) + r^2 v(\theta) + \beta_2 r^2 + O(r^3)$, we collect terms by powers of $r$:
\begin{itemize}
    \item Order $r$: $u(\theta + \alpha) - u(\theta) = -\cos\theta$. 
    Expressing $\cos\theta = \frac{e^{i\theta} + e^{-i\theta}}{2}$, we solve this via the ansatz $u(\theta) = \operatorname{Re}(a e^{i\theta})$, obtaining:
    \[
    a = -\frac{1}{e^{i\alpha}-1} = \frac{1}{2} + \frac{i}{2}\cot\left(\frac{\alpha}{2}\right).
    \]
    \item Order $r^2$: $u'(\theta)\cos\theta + v(\theta+\alpha) - v(\theta) = \beta_2$. 
    Averaging over the period $[0, 2\pi]$ (noting that the telescoping sum of $v$ vanishes in mean), we find that $\beta_2$ equals the mean value of $u'(\theta)\cos\theta$:
    \[
    \beta_2 = \frac{1}{2\pi} \int_0^{2\pi} u'(\theta)\cos\theta \, \dd\theta = -\frac{1}{2}\operatorname{Im}(a e^{i\alpha}) = -\frac{1}{4}\cot\left(\frac{\alpha}{2}\right) = -\frac{1}{4}\cot(\pi\omega_0).
    \]
\end{itemize}
Since rotation numbers are conjugacy invariants and the lift of $h_r \circ g_r \circ h_r^{-1}$ shifts by $\alpha + \beta_2 r^2 + O(r^3)$, the rotation number satisfies:
\[
\rho(r) = \omega_0 - \frac{r^2}{8\pi}\cot(\pi\omega_0) + O(r^3).
\]
Because $\omega_0 \in (0, \tfrac12)$, we have $\cot(\pi\omega_0) > 0$. Thus, $\rho(r)$ is strictly monotonic, in a punctured neighbourhood of $0$. By continuity, for every $\delta > 0$, there exists an $r_* \in (0, \delta)$ such that $\rho(r_*) = \frac{p'}{q'} \in \Q$.

To analyse the map $G(\theta) = g_{r_*}(\theta) = \theta + \alpha + r_*\cos\theta$, we extend it to the complex plane and define the complex map:
\[
G(z) = z + \alpha + r_* \cos z.
\]
Both this map and its iterate $G^{q'}(z)$ are transcendental entire functions, provided $r_* \neq 0$. If a map like $G^{q'}$ was globally injective on $\mathbb{C}$, it would have to be an affine function. SSince it is not (due to the presence of trigonometric terms), we conclude that  $G^{q'}$ is not globally injective.

If the following relation holds:
\[
G^{q'}(\theta) = \theta + 2\pi p' \quad \text{for all } \theta \in \mathbb{R},
\]
and since $G^{q'}(z) - z - 2\pi p'$ is an analytic function that vanishes identically on the real axis, the Identity Theorem for analytic functions would force this equality to hold identically all over $\mathbb{C}$:
\[
G^{q'}(z) = z + 2\pi p' \quad \text{for all } z \in \mathbb{C},
\]
which would imply that $G^{q'}$ is a global translation, thus injective.

This proves that the circle map $g_{r_*}$ possesses only a finite set of periodic orbits of period $q'$.

Finally, let us suppose that there exists a local homeomorphism $h$ conjugating $g$ to a twist map $f_\nu(r,\theta) = (r, \theta + \Omega(r^2))$ in a neighbourhood of $0$, with $h(0)=0$. 
Choose $\delta >0$ sufficiently small and $r_* \in (0, \delta)$ such that $\rho(r_*) = p'/q'$. Then $K = h(C_{r_*})$ forms a Jordan curve invariant under $f_\nu$ surrounding the origin. 
Let $s_1 = \min_{w\in K}|w|$ and $s_2 = \max_{w\in K}|w|$, $K$ intersects every intermediate circle $C_s$ for $s \in [s_1, s_2]$.
\begin{enumerate}
    \item If $\Omega$ is non-constant on the interval $[s_1^2, s_2^2]$, there exists a radius $s \in (s_1, s_2)$ such that the twist angle $\Omega(s^2)/(2\pi) \notin \Q$. On this circle $C_s$, the map $f_\nu$ acts as an irrational rotation. However, $K \cap C_s$ is a non-empty, compact, $f_\nu$-invariant subset, forcing $K \supset C_s$. A Jordan curve containing another concentric Jordan curve must coincide with it, contradicting $s_1 < s_2$.
    \item If $\Omega$ is identically constant equal to $c$ on $[s_1^2, s_2^2]$, then $f_\nu|_K$ acts as a rigid rotation $R_c$. Because $h$ is a conjugacy mapping periodic points bijectively, $K$ must contain a periodic point of period $q'$ (originating from the periodic points of $g_{r_*}$). This implies $R_c^{q'} = \mathrm{id}$, which in turn forces $g_{r_*}^{q'} = \mathrm{id}$, directly contradicting the finiteness of the periodic points established above.
\end{enumerate}

Thus, such a topological conjugacy cannot exist.
\end{proof}

\begin{remark}
Proposition~\ref{prop:counter} describes a phenomenon that cannot occur in the symplectic setting. Along each invariant curve of the first integral the dynamics is a circle diffeomorphism, which in general mode-locks at rational rotation numbers. Area preservation rules this out: by Theorem~\ref{thm:sympl} it forces the dynamics on each level curve to be a rigid rotation. The proof of Proposition~\ref{prop:counter} also shows that no analogue of Zung's convergence theorem \cite{Zung} holds for planar maps with one analytic first integral unless an area form is preserved.
\end{remark}

\section{Maps with two first integrals}\label{sec:two}

We now turn our attention to the implications of the existence of a second first integral

\begin{theorem}\label{thm:main}
Assume \textup{(H)}, and let $I_2$ be a continuous first integral of $f$ defined on a neighbourhood of $0$. Then at least one of the following holds:
\begin{enumerate}[label=\textup{(\alph*)}]
\item $I_2=H\circ I_1$ on a neighbourhood of $0$, for some continuous function $H$;
\item $f^N=\mathrm{id}$ near $0$, for some $N\ge1$.
\end{enumerate}
In case \textup{(b)}, $\omega_0=p/q\in\Q$, $f^q=\mathrm{id}$, and $f$ is real-analytically conjugate to its linear part. If, moreover, $f$ preserves an analytic area form, the conjugacy can be chosen symplectic.
\end{theorem}

\begin{proof}
We work in the coordinates of Proposition~\ref{prop:circle} on a small disc $D_\varepsilon$, and write $\tilde I_2=I_2\circ G_0^{-1}$.

Let us prove that, if $g_r$ is not of finite order, then $\tilde I_2$ is constant on $C_r$.

Let us suppose first that the rotation number of $g_r$ is irrational. Since $g_r$ is analytic, hence $C^2$, Denjoy's theorem \cite{KH} shows that every orbit of $g_r$ is dense in $C_r$. Since $\tilde I_2$ is continuous and constant along orbits, it is constant on $C_r$.

Suppose next that the rotation number is $p'/q'$ but $g_r^{q'}\neq\mathrm{id}$. Then the fixed points of $g_r^{q'}$ are the, finitely many, zeros of a non-zero analytic function. On each arc between two consecutive fixed points, $g_r^{q'}$ moves points monotonically, so every orbit of $g_r^{q'}$ in the arc converges to one endpoint in forward time and to the other in backward time. Invariance and continuity force $\tilde I_2$ to take on the arc the same value as at both of its endpoints. Since these arcs cover the circle, $\tilde I_2$ is constant on $C_r$.

Now, for $N\ge1$, the set $P_N=\{r\in(0,\varepsilon):g_r^N=\mathrm{id}\}$ is closed in $(0,\varepsilon)$. If some $P_N$ has non-empty interior, then $g^N=\mathrm{id}$ on an open annulus, hence on $D_\varepsilon$ by analyticity, and we are in the case (b) of the statement. Otherwise $\bigcup_NP_N$ is meagre, and $\tilde I_2$ is constant on $C_r$ for a dense set of radii $r$. For fixed $\theta_1,\theta_2$, the continuous function $r\mapsto\tilde I_2(r,\theta_1)-\tilde I_2(r,\theta_2)$ then vanishes on a dense set, hence everywhere. Thus $\tilde I_2$ is a function of $r$ alone, and we are in case (a).

Let us now suppose that $g_r$ is of finite order. If $g^N=\mathrm{id}$, then $(\dd_0g)^N=\mathrm{id}$, so $\mu^N=1$, $\omega_0=p/q$ and $q\mid N$. Let $\Lambda=\dd_0g$. The average
\[
h=\frac1N\sum_{k=0}^{N-1}\Lambda^{-k}\circ g^k
\]
is real-analytic, satisfies $\dd_0h=\mathrm{id}$, and $h\circ g=\Lambda\circ h$. Hence $g$, and therefore $f$, is real-analytically conjugate to its linear part, and $f^q=\mathrm{id}$ because $\Lambda^q=\mathrm{id}$.

If $f$ preserves an analytic area form, we may assume by Remark~\ref{rem:area} that it preserves $\omega$, and use the coordinates of Theorem~\ref{thm:sympl} instead. There, $g^N=\mathrm{id}$ means $N\,\Omega(s)\in2\pi\Z$ for all small $s$, so $\Omega$ is constant. Then $S$ already conjugates $f$ symplectically to a rotation, which is symplectically conjugate to $L$.
\end{proof}

\begin{corollary}\label{cor:independent}
Assume \textup{(H)}, and let $I_2$ be a $C^1$ first integral of $f$ (for instance smooth, real-analytic, or real-meromorphic without a pole at $0$) such that $\dd I_1\wedge\dd I_2\neq0$ at points arbitrarily close to $0$. Then $\omega_0=p/q\in\Q$, $f^q=\mathrm{id}$, and $f$ is real-analytically (and, if $f$ preserves an analytic area form, symplectically) conjugate to its linear part.

Conversely, if $\omega_0=p/q$ and $f$ is real-analytically conjugate to its linear part, then $f$ admits a real-analytic first integral that is functionally independent of $I_1$ on an open dense set; in linearising coordinates one can take $\operatorname{Re}(w^q)$.
\end{corollary}

\begin{proof}
In case (a) of Theorem~\ref{thm:main}, $\tilde I_2$ is constant on the level curves of $|w|^2$, so $\dd\tilde I_2$ annihilates their tangent vectors and $\dd I_1\wedge\dd I_2=0$ wherever $I_2$ is differentiable, near $0$. Hence case (b) holds. For the converse, in linearising coordinates $f$ is a rotation by $\pm2\pi p/q$, which preserves $\operatorname{Re}(w^q)$. Suppose that $\dd I_1\wedge\dd\operatorname{Re}(w^q)$, which is analytic, vanished identically near $0$. Then $I_1$ would be locally constant along the level curves of $\operatorname{Re}(w^q)$ away from $0$, in particular along the rays $\{\cos q\theta=0\}$ emanating from $0$. By continuity, $I_1$ would vanish on these rays, contradicting the fact that $0$ is a strict minimum of $I_1$. Hence the two integrals are independent on an open dense set.
\end{proof}

\begin{corollary}\label{cor:irr}
If $\omega_0\notin\Q$, then under \textup{(H)} every continuous first integral of $f$ is, near $0$, a function of $I_1$. In particular, $f$ admits no second first integral that is $C^1$ and functionally independent of $I_1$ at points arbitrarily close to $0$. This holds whether or not $f$ is symplectic.
\end{corollary}

\begin{corollary}\label{cor:sympl-two}
Assume \textup{(H)} and that $f$ preserves an analytic area form, and let $\Omega$ be as in Theorem~\ref{thm:sympl}. Then $f$ admits a $C^1$ second first integral, independent of $I_1$ at points arbitrarily close to $0$, if and only if $\omega_0=p/q\in\Q$ and $\Omega$ is constant, i.e.\ if and only if $f$ is symplectically conjugate to a rotation by $\pm2\pi p/q$.
\end{corollary}

\begin{remark}
For a smooth second integral, no condition on its Taylor series is needed, because the argument above never uses power series. Second integrals that are real-meromorphic with a pole at the origin are not covered, since they are not continuous at $0$.
\end{remark}

\section{An example: the McMillan map}\label{sec:example}

Let $\mu\in(-1,1)$, write $\mu=\cos\alpha$ with $\alpha\in(0,\pi)$, and consider the McMillan map \cite{McMillan}, a symmetric QRT map (see  \cite{QRT1988,QRT1989,Duistermaat}):
\[
f_\mu(x,y)=\Big(y,\ -x+\frac{2\mu y}{1+y^2}\Big),\qquad I_\mu(x,y)=x^2y^2+x^2+y^2-2\mu xy .
\]
With $y'=-x+2\mu y/(1+y^2)$ one has $y'(1+y^2)-2\mu y=-x(1+y^2)$, from which $I_\mu(y,y')=x^2(1+y^2)-2\mu xy+y^2=I_\mu(x,y)$. Moreover $\det Df_\mu\equiv1$. The linear part has trace $2\mu$ and determinant $1$, so the origin is elliptic with rotation number $\omega_0=\alpha/2\pi\in(0,\tfrac12)$. The Hessian of $I_\mu$ at $0$, namely $2\begin{psmallmatrix}1&-\mu\\-\mu&1\end{psmallmatrix}$, is positive definite. Hence (H) holds and $f_\mu$ is area-preserving.

\begin{proposition}\label{prop:mcm}
\begin{enumerate}[label=\textup{(\roman*)}]
\item For every $\mu\in(-1,1)$, $f_\mu$ is real-analytically and symplectically conjugate to a twist map $w\mapsto w\,e^{i\Omega_\mu(|w|^2)}$.
\item Let $\omega(h)$ denote the rotation number of $f_\mu$ on the invariant curve $\{I_\mu=h\}$. Then
\[
2\pi\omega(h)=\alpha+\frac{3\cos\alpha}{4\sin^3\alpha}\,h+O(h^2).
\]
\item For $\mu\neq0$, $\Omega_\mu$ is not constant. Consequently, every continuous first integral of $f_\mu$ is, near $0$, a function of $I_\mu$. This holds for both irrational and for resonant values of $\omega_0$.
\item For $\mu=0$, $f_0(x,y)=(y,-x)$ is linear with $\omega_0=\tfrac14$, and $x^4-6x^2y^2+y^4=\operatorname{Re}(x+iy)^4$ is a second real-analytic first integral, independent of $I_0$.
\end{enumerate}
\end{proposition}

\begin{proof}
(i) This is Theorem~\ref{thm:sympl}.

(ii) Put $\lambda=e^{i\alpha}$ and $(x,y)=(z+\bar z,\ \lambda z+\bar\lambda\bar z)$. In the coordinate $z$ the linear part of $f_\mu$ is $z\mapsto\lambda z$, and $z'=(y'-\bar\lambda x')/(\lambda-\bar\lambda)$. Since $f_\mu$ is odd, it has no quadratic terms. Its cubic part comes from $2\mu y/(1+y^2)=2\mu y-2\mu y^3+O(y^5)$, and the coefficient of the resonant monomial $z^2\bar z$ in $z'$ equals
\[
\frac{-2\mu\cdot3\lambda}{2i\sin\alpha}=3i\cot\alpha\cdot\lambda .
\]
Because neither $f_\mu$ nor the twist map of (i) has quadratic terms, the quadratic part of any conjugacy between them does not contribute at cubic order. Hence the resonant cubic coefficient is a conjugacy invariant, up to the scaling introduced by the linear part of the conjugacy. (For $\omega_0=\tfrac13$ the additional resonant monomial $\bar z^2$ is quadratic and plays no role.) Comparing with $w\,e^{i\Omega(|w|^2)}=e^{i\Omega(0)}w\big(1+i\Omega'(0)|w|^2+\dots\big)$, and using $I_\mu=4\sin^2\!\alpha\,|z|^2+O(|z|^4)$, we obtain
\[
\frac{\dd}{\dd h}\Big|_{h=0}2\pi\omega(h)=\frac{3\cot\alpha}{4\sin^2\alpha},
\]
which is (ii).

(iii) By (ii), $\omega$ is not constant when $\cos\alpha\neq0$, so $\Omega_\mu$ is not constant, and the claim follows from Theorem~\ref{thm:main} and Corollary~\ref{cor:sympl-two}.

(iv) The map $(x,y)\mapsto(y,-x)$ acts as $x+iy\mapsto-i(x+iy)$ and therefore preserves $(x+iy)^4$. Independence from $I_0$ follows from Corollary~\ref{cor:independent}.
\end{proof}

The McMillan family thus realises both alternatives of Theorem~\ref{thm:main} within a single one-parameter family: the generic case $\mu\neq0$, where $I_\mu$ is the only first integral, and the exceptional linear case $\mu=0$. Proposition~\ref{prop:counter} shows what may happen when area preservation is dropped.

\section{A reduction of a Yang--Baxter map}\label{sec:yb}

Yang--Baxter maps \cite{Veselov2003,SurisVeselov,PTV,KouloukasPapageorgiou2009,KouloukasPapageorgiou2011} are another rich source of integrable maps. Suppose a Yang--Baxter map $R_{a,b}:(x,y)\mapsto(u,v)$ has a Lax matrix $L(x;a,\lambda)$, i.e.
\[
L(u;a,\lambda)\,L(v;b,\lambda)=L(y;b,\lambda)\,L(x;a,\lambda).
\]
Then the spectral invariants of $L(y;b,\lambda)L(x;a,\lambda)$ are first integrals of the iterates of $R_{a,b}$ \cite{Veselov2003}. Indeed, $L(v;b,\lambda)L(u;a,\lambda)$ is conjugate to $L(u;a,\lambda)L(v;b,\lambda)=L(y;b,\lambda)L(x;a,\lambda)$.

We use the Darboux matrix of the nonlinear Schr\"odinger equation \cite{KRMikhailov}, a Lax matrix that is a first-degree polynomial in $\lambda$, of the type studied in \cite{KouloukasPapageorgiou2009,KouloukasPapageorgiou2011}:
\[
M(p,q,\chi;\lambda)=\lambda\begin{pmatrix}1&0\\0&0\end{pmatrix}+\begin{pmatrix}\chi+pq&p\\q&1\end{pmatrix}.
\]
Consider the refactorisation problem:
\[
M(P_1,Q_1,a;\lambda)\,M(P_2,Q_2,b;\lambda)=M(p_2,q_2,b;\lambda)\,M(p_1,q_1,a;\lambda).
\]
Comparing the coefficients of $\lambda$ gives $P_2=p_1$ and $Q_1=q_2$. The terms of order $\lambda^0$ are then linear in $P_1,Q_2$, and the problem has the unique solution:
\begin{equation}\label{eq:AY}
R_{a,b}:\quad P_1=p_2+\frac{(b-a)\,p_1}{1+p_1q_2},\qquad Q_1=q_2,\qquad P_2=p_1,\qquad Q_2=q_1+\frac{(a-b)\,q_2}{1+p_1q_2}.
\end{equation}
This is the Adler--Yamilov map \cite{AdlerYamilov,KouloukasPapageorgiou2009,KRMikhailov}. It satisfies the Yang--Baxter equation, and the iterates of $R_{a,b}$ preserve:
\[
H_1=p_1q_1+p_2q_2,\qquad H_2=(b+p_2q_2)(a+p_1q_1)+p_1q_2+p_2q_1,
\]
which are, up to additive constants, the coefficients of $\lambda$ and $\lambda^0$ in $\operatorname{tr}M(p_2,q_2,b;\lambda)M(p_1,q_1,a;\lambda)$. We reduce $R_{a,b}$ to a planar map in two steps.

\begin{lemma}\label{lem:real}
Let $\varepsilon\in\{1,-1\}$, $\alpha,\beta\in\R$, $a=\alpha-\tfrac{i\beta}{2}$ and $b=\alpha+\tfrac{i\beta}{2}$. The set:
\[
\Sigma_\varepsilon=\{q_1=\varepsilon\bar p_2,\ q_2=\varepsilon\bar p_1\}\subset\C^4
\]
is invariant under $R_{a,b}$. In the coordinates $(p_1,p_2)\in\C^2$ on $\Sigma_\varepsilon$, $R_{a,b}$ becomes:
\[
\rho(p_1,p_2)=\Big(p_2+\frac{i\beta\,p_1}{1+\varepsilon|p_1|^2},\ p_1\Big),
\]
so that its orbits are the solutions $(p_1,p_2)=(p_n,p_{n-1})$ of the lattice equation:
\begin{equation}\label{eq:lattice}
p_{n+1}-p_{n-1}=\frac{i\beta\,p_n}{1+\varepsilon|p_n|^2}.
\end{equation}
On $\Sigma_\varepsilon$, $H_1=2\varepsilon\operatorname{Re}(p_1\bar p_2)$ and $H_2=|a+\varepsilon p_1\bar p_2|^2+\varepsilon(|p_1|^2+|p_2|^2)$.
\end{lemma}

\begin{proof}
On $\Sigma_\varepsilon$ the denominator $1+p_1q_2=1+\varepsilon|p_1|^2$ is real, and $b-a=i\beta$. Hence $Q_1=q_2=\varepsilon\bar p_1=\varepsilon\bar P_2$, and:
\[
Q_2=\varepsilon\bar p_2-\frac{i\beta\,\varepsilon\bar p_1}{1+\varepsilon|p_1|^2}=\varepsilon\bar P_1 .
\]
The formulas for $H_1,H_2$ follow by substituting $p_1q_1=\varepsilon p_1\bar p_2$, $p_2q_2=\varepsilon\overline{p_1\bar p_2}$, $p_1q_2=\varepsilon|p_1|^2$, $p_2q_1=\varepsilon|p_2|^2$, and using $b=\bar a$.
\end{proof}

The map $\rho$ commutes with the $U(1)$ action $p_j\mapsto e^{i\vartheta}p_j$. Its invariants are $A=|p_1|^2$, $B=|p_2|^2$ and $Z=p_1\bar p_2=X+iY$, subject to the relation $X^2+Y^2=AB$.

\begin{proposition}\label{prop:yb-red}
Under $\rho$ we have $X'=X$, $B'=A$, and
\[
Y'=-Y+\frac{\beta A}{1+\varepsilon A},\qquad A'=\frac{X^2+Y'^2}{A}.
\]
Fix $X=c\neq0$, that is, $H_1=2\varepsilon c$. On the reduced space $\{AB=c^2+Y^2\}$, parametrised by $(A,Y)$ with $A>0$ and $1+\varepsilon A\neq0$, the map $\rho$ induces:
\begin{equation}\label{eq:T}
T(A,Y)=\Big(\frac{c^2+Y'^2}{A},\ Y'\Big),\qquad Y'=-Y+\frac{\beta A}{1+\varepsilon A}.
\end{equation}
Moreover $T=\iota_2\circ\iota_1$, where:
\[
\iota_1(A,Y)=\Big(A,\ -Y+\frac{\beta A}{1+\varepsilon A}\Big),\qquad \iota_2(A,Y)=\Big(\frac{c^2+Y^2}{A},\ Y\Big),
\]
are involutions which preserve:
\begin{equation}\label{eq:K}
K(A,Y)=Y^2-\varepsilon\beta Y+\varepsilon\Big(A+\frac{c^2+Y^2}{A}\Big)
\end{equation}
and reverse the area form $A^{-1}\dd A\wedge\dd Y$. Consequently $T$ preserves $K$ and $A^{-1}\dd A\wedge\dd Y$; in the coordinates $(\log A,Y)$ it is area-preserving. The level sets of $K$ are the biquadratic curves $AY^2-\varepsilon\beta AY+\varepsilon(A^2+Y^2+c^2)=\kappa A$, so $T$ is a map of QRT type. Finally, on $\Sigma_\varepsilon\cap\{X=c\}$ we have $H_2=K+|a|^2+2\varepsilon\alpha c+c^2$.
\end{proposition}

\begin{proof}
Write $\gamma=i\beta/(1+\varepsilon A)$, which is purely imaginary. Then $P_1\bar P_2=(p_2+\gamma p_1)\bar p_1=\bar Z+\gamma A$, which gives $X'$ and $Y'$. Further, $B'=|P_2|^2=A$, and $A'B'=|Z'|^2$ gives $A'$.

For fixed $A$, $K$ is quadratic in $Y$, with leading coefficient $1+\varepsilon/A$ and linear coefficient $-\varepsilon\beta$. Hence $K$ is invariant under $Y\mapsto\varepsilon\beta/(1+\varepsilon/A)-Y=\beta A/(1+\varepsilon A)-Y$, which is $\iota_1$. For fixed $Y$, the expression $A+(c^2+Y^2)/A$ is invariant under $A\mapsto(c^2+Y^2)/A$, which is $\iota_2$.

For the area form: $\iota_1$ fixes $A$ and $\iota_1^*\dd Y=-\dd Y+(\ast)\,\dd A$, so $\iota_1^*(A^{-1}\dd A\wedge\dd Y)=-A^{-1}\dd A\wedge\dd Y$. For $\iota_2$, with $A'=(c^2+Y^2)/A$, we have $\dd A'\wedge\dd Y=-(A'/A)\,\dd A\wedge\dd Y$, so again $\iota_2^*(A^{-1}\dd A\wedge\dd Y)=-A^{-1}\dd A\wedge\dd Y$.

The expression for $H_2$ follows from Lemma~\ref{lem:real}, using $\operatorname{Re}(\bar aZ)=\alpha c-\beta Y/2$ and $|Z|^2=c^2+Y^2$.
\end{proof}

The fixed points of $T$ have a transparent meaning: they are the plane waves of \eqref{eq:lattice}.

\begin{proposition}\label{prop:yb-fixed}
Let $\varepsilon=1$, $A_*>0$ and $\varphi\in(0,\tfrac\pi2)$, and set:
\[
\beta=2(1+A_*)\sin\varphi,\qquad c=A_*\cos\varphi .
\]
Then $p_n=\sqrt{A_*}\,e^{in\varphi}$ solves \eqref{eq:lattice}, and corresponds to the fixed point $(A_*,Y_*)=(A_*,A_*\sin\varphi)$ of $T$. This fixed point is elliptic, with rotation number $\omega_0$ given by:
\begin{equation}\label{eq:yb-omega0}
\cos2\pi\omega_0=\frac{2\sin^2\varphi}{1+A_*}-1,
\end{equation}
and $K$ has a non-degenerate minimum there. Every $\omega_0\in(0,\tfrac12)$ occurs for suitable $(A_*,\varphi)$.
\end{proposition}

\begin{proof}
For the plane wave, $p_{n+1}-p_{n-1}=2i\sin\varphi\,p_n=i\beta p_n/(1+A_*)$. Along it, $A=A_*$ and $Z=p_n\bar p_{n-1}=A_*e^{i\varphi}$, so $X=c$ and $Y=A_*\sin\varphi$ are constant.

With $g(A)=\beta A/(1+A)$, the Jacobian of $T$ at the fixed point is:
\[
DT=\begin{pmatrix}2Y_*g'(A_*)/A_*-1&-2Y_*/A_*\\ g'(A_*)&-1\end{pmatrix},
\]
with determinant $1$ and trace $2Y_*g'(A_*)/A_*-2=\beta^2/(1+A_*)^3-2=4\sin^2\varphi/(1+A_*)-2\in(-2,2)$.

The first derivatives $K_A=1-(c^2+Y^2)/A^2$ and $K_Y=2Y-\beta+2Y/A$ vanish at $(A_*,Y_*)$. The second derivatives there are $K_{AA}=2/A_*$, $K_{AY}=-2Y_*/A_*^2$ and $K_{YY}=2+2/A_*$, and the Hessian determinant is $4(A_*^3+c^2)/A_*^4>0$.

Given $\omega_0\in(0,\tfrac12)$, choose $A_*<2/(1+\cos2\pi\omega_0)-1$ and then $\varphi$ with $\sin^2\varphi=\tfrac12(1+A_*)(1+\cos2\pi\omega_0)$.
\end{proof}

In the coordinates $(\log A-\log A_*,\,Y-Y_*)$, the map $T$ and the integral $K-K(A_*,Y_*)$ therefore satisfy (H), and $T$ is area-preserving. The results of Sections~\ref{sec:one} and~\ref{sec:two} give the following:

\begin{proposition}\label{prop:yb}
Let $\varepsilon=1$ and let $(A_*,\varphi)$ be as in Proposition~\ref{prop:yb-fixed}.
\begin{enumerate}[label=\textup{(\roman*)}]
\item Near $(A_*,Y_*)$, the map $T$ is real-analytically conjugate, by a map preserving $A^{-1}\dd A\wedge\dd Y$, to a twist map $w\mapsto w\,e^{i\Omega(|w|^2)}$.
\item If $\omega_0\notin\Q$, every continuous first integral of $T$ is, near $(A_*,Y_*)$, a function of $K$.
\item If $\omega_0=p/q$, then $T$ has a $C^1$ first integral independent of $K$ near $(A_*,Y_*)$ if and only if $T^q=\mathrm{id}$. For $(A_*,\sin\varphi)=(\tfrac54,\tfrac34)$, $(\tfrac18,\tfrac34)$ and $(\tfrac2{25},\tfrac9{10})$, which give $\omega_0=\tfrac13$, $\tfrac14$ and $\tfrac16$ respectively, this does not happen. Hence these resonant plane waves admit no second first integral.
\end{enumerate}
\end{proposition}

\begin{proof}
(i) follows from Theorem~\ref{thm:sympl} and Remark~\ref{rem:area}, and (ii) from Corollary~\ref{cor:irr}. The equivalence in (iii) is Theorem~\ref{thm:main} together with Corollary~\ref{cor:independent}.

For the three parameter sets in (iii), the quantities $\beta$ and $c^2=A_*^2\cos^2\varphi$ are rational: $(\beta,c^2)=(\tfrac{27}8,\tfrac{175}{256})$, $(\tfrac{27}{16},\tfrac7{1024})$ and $(\tfrac{243}{125},\tfrac{19}{15625})$. An exact computation in rational arithmetic shows that $T^q(A_*+\tfrac1{100},Y_*)\neq(A_*+\tfrac1{100},Y_*)$. Since $T$ maps the half-plane $\{A>0\}$ analytically into itself, $T^q-\mathrm{id}$ is real-analytic on this connected set. Being non-zero at one point, it cannot vanish identically near $(A_*,Y_*)$.
\end{proof}

\begin{remark}
\begin{enumerate}[label=(\roman*)]
\item For $\varepsilon=-1$ the same computations, on $\{0<A<1\}$, give plane waves $p_n=\sqrt{A_*}e^{in\varphi}$ with $\beta=2(1-A_*)\sin\varphi$. The trace of $DT$ is now $4\sin^2\varphi/(1-A_*)-2$, and the Hessian determinant of $K$ is $4(c^2-A_*^3)/A_*^4=4(\cos^2\varphi-A_*)/A_*^2$. Thus the plane wave is elliptic, with $K$ having a non-degenerate maximum, if $A_*<\cos^2\varphi$, and hyperbolic if $A_*>\cos^2\varphi$: a discrete analogue of modulational instability. In the elliptic case the conclusions of Proposition~\ref{prop:yb}(i)--(ii) hold verbatim.
\item In terms of the original variables, Proposition~\ref{prop:yb} describes the dynamics of the real Adler--Yamilov map near its relative equilibria, the plane waves, modulo the phase symmetry. The twist map in (i) is the reduced dynamics on the invariant tori surrounding them.
\end{enumerate}
\end{remark}

\section{Conclusions}

We have studied the effect that the existence of first integrals has on the normal form of a real-analytic planar map at an elliptic fixed point. When the map preserves an analytic area form, a single analytic first integral with a non-degenerate extremum forces the map to be analytically and symplectically conjugate to a twist map, for every rotation number. Without area preservation this fails: the dynamics on the invariant curves of the integral can mode-lock, and the Poincar\'e normalisation then diverges despite the existence of an analytic first integral. In both settings, a second, functionally independent first integral of any regularity class (continuous in the sense of Theorem~\ref{thm:main}, $C^1$, smooth, real-meromorphic without pole at the origin, or real-analytic) exists only if the map is periodic, in which case it is linearisable. In particular, no such integral exists for irrational rotation numbers. The McMillan map and the reduced Adler--Yamilov map show that these results apply directly to maps arising in discrete integrable systems. In the latter, the elliptic fixed points are the plane waves of the reduced lattice equation, and their stability ($\varepsilon=1$) or modulational instability ($\varepsilon=-1$, $A_*>\cos^2\varphi$) is read off directly.

Several questions remain open. In the resonant non-symplectic case, we do not know whether a map with an analytic first integral is always analytically conjugate to a circle-preserving map commuting with its linear part. Second first integrals having a pole at the origin were not considered here. Finally, it would be interesting to extend these results to higher dimensions, in the framework of integrable diffeomorphisms \cite{JiangStolovitch}. We hope that these results will be useful in the analysis of specific integrable mappings, such as the QRT family and its generalisations.
\section*{Acknowledgements}
The author would like to thank Sotiris Konstantinou-Rizos for spotting a critical error in an earlier version of this paper and bringing to his attention the two maps analysed here.

\end{document}